\documentclass[11pt]{amsart}
\usepackage[top=3cm, bottom=3.5cm, left=3cm, right=3cm]{geometry}
\usepackage{amsmath,amsfonts,amssymb,amsthm}
\usepackage{mathtools}
\usepackage{tikz-cd}
\usepackage[utf8]{inputenc}
\usepackage{hyperref}
\usepackage{enumerate}
\usepackage[utf8]{inputenc} 
\usepackage{mathrsfs,amsmath}
\usepackage{amsfonts}
\usepackage{tikz-cd}

\usepackage{geometry} 
\usepackage{graphicx} 

 \usepackage[parfill]{parskip} 

\usepackage{booktabs} 
\usepackage{array} 
\usepackage{paralist} 
\usepackage{verbatim} 
\usepackage{subfig} 

\usepackage{fancyhdr} 
\author{Xi Sisi Shen}
\address{Department of Mathematics\\
  Northwestern University\\
  Evanston, IL USA 60208}
\email[X. S. Shen]{xss@math.columbia.edu}

\newtheorem{thm}{Theorem}

\newtheorem{lemma}{Lemma}

\DeclareMathOperator\Ree{Re}
\DeclareMathOperator\Ent{Ent}

\DeclareMathOperator\Mab{Mab}
\DeclareMathOperator\tr{tr}

\DeclareMathOperator\Ric{Ric}
\DeclareMathOperator\R{R}

\DeclareMathOperator\BC{BC}

\begin{document}
\bibliographystyle{amsplain}

\title{A Chern-Calabi flow on Hermitian manifolds}

\begin{abstract}
We study an analogue of the Calabi flow in the non-K\"ahler setting for compact Hermitian manifolds with vanishing first Bott-Chern class. We prove a priori estimates for the evolving metric along the flow given a uniform bound on the Chern scalar curvature. If the Chern scalar curvature remains uniformly bounded for all time, we show that the flow converges smoothly to the unique Chern-Ricci-flat metric in the $\partial\bar{\partial}$-class of the initial metric.
\end{abstract}
\maketitle
\section{Introduction}
The Calabi flow was introduced by Calabi for K\"ahler metrics in \cite{calabi1,calabi2} and is defined by
\begin{align*}
\frac{\partial \omega}{\partial t}  = \sqrt{-1}\partial\bar{\partial} \R, \ \ \omega(0) = \omega_0,
\end{align*}
where $\R = \tr_\omega \Ric(\omega) = \tr_\omega \sqrt{-1}\partial\bar{\partial}\log\det g$ is the scalar curvature of a K\"ahler metric $g$, with associated K\"ahler form $\omega$, and the flow preserves the K\"ahler class of the metric. If we let $\omega_\varphi(t) = \omega_0 + \sqrt{-1}\partial\bar{\partial}\varphi(t)$ and normalize $\varphi\in C^\infty(X)$ such that $\int_X \varphi \omega_\varphi^n = 0$, then the flow can be represented in terms of the potential function $\varphi$ by
\begin{align*}
\frac{\partial \varphi}{\partial t}  = \R_\varphi - \underline{\R},
\end{align*}
where $\R_\varphi$ is the scalar curvature of $\omega_\varphi $ and $\underline{\R}=\frac{\int_X \R_\varphi \omega_\varphi^n}{\int_X \omega_\varphi^n}$ is the average scalar curvature of $\omega_\varphi$ on $X$ which is independent of $t$.
Short-time existence of the Calabi flow follows from the fact that it is a fourth order quasilinear parabolic equation as shown by Chen-He \cite{chen-he08}. They also proved the global existence of the flow under the assumption of a uniform Ricci curvature bound \cite{chen-he08}. It was shown by Sz\'ekelyhidi in \cite{szekelyhidi12} that if the curvature tensor is uniformly bounded along the Calabi flow and the Mabuchi energy is proper, then the flow converges to a constant scalar curvature K\"ahler (cscK) metric. Chen-Sun proved in \cite{chen-sun10} that if the Calabi flow starts at an initial metric that is sufficiently close to a cscK metric, then the flow exists and converges uniformly to the cscK metric. Additional literature on the Calabi flow can be found in \cite{bv20, bdl, chang, chrusciel, hf12, fine10, he15, hz12, lwz, streets14, streets16, tw07}. 

We define an analogue of the Calabi flow on the $\partial\bar{\partial}$-class of Hermitian metrics when the first Bott-Chern class vanishes. Let $(X,\omega_0)$ be a Hermitian manifold and let $g_0$ be the associated Hermitian metric to the real $(1,1)$-form $\omega_0$. The real $(1,1)$ Bott-Chern cohomology is defined as
\begin{align*}
H^{1,1}_{\BC}(X,\mathbb{R}) = \frac{\{\text{$d$-closed real (1,1)-forms}\}}{\{\sqrt{-1}\partial\bar{\partial}\psi, \ \psi\in C^\infty(X,\mathbb{R}) \}}
\end{align*} 
and the first Bott-Chern class, denoted $c_1^{\BC}(X)$, is given by the $\partial\bar{\partial}$-class of the Chern-Ricci form, $$\Ric(\omega)=-\sqrt{-1}\partial\bar{\partial}\log\det g,$$ for any metric $\omega$ on $X$. Now, let us define the space of metrics $$\mathcal{H} = \{\omega_\varphi : \omega_\varphi=\omega_0 + \sqrt{-1}\partial\bar{\partial}\varphi >0,  \varphi \in C^\infty(X,\mathbb{R})\}.$$  In the setting of vanishing first Bott-Chern class, $c_1^{\BC}(X)=0$, Tosatti-Weinkove \cite{tw13} observed that one can define the Mabuchi energy $\Mab_{\omega_0}(\omega_\varphi): \mathcal{H}\rightarrow \mathbb{R}$ by
\begin{align*}
\Mab_{\omega_0}(\omega_\varphi) = \frac{1}{V}\int_X \Big(\log\frac{\det g_\varphi}{\det g_0} - F \Big)\omega_\varphi^n + \frac{1}{V}\int_X F\omega_0^n,
\end{align*}
where $F$ is the Chern-Ricci potential of $\omega_0$, that is $\Ric(\omega_0) = \sqrt{-1}\partial\bar{\partial}F$, normalized so that $\int_X e^F\omega_0^n = \int_X \omega_0^n$. This definition of Mabuchi energy agrees with the formula in the K\"ahler setting \cite{tian2000} (see also Section 9 of \cite{tw13}). Similar to how there are several generalizations of the K\"ahler-Ricci flow to the non-K\"ahler setting \cite{gill11, streets_tian12,tw15, yury}, the Chern-Calabi flow we consider in this paper may not be the only generalization of the Calabi flow evolving the $\partial\bar{\partial}$-potential of the metric by its Chern scalar curvature. In \cite{bv20}, Bedulli-Vezzoni prove short-time existence of a Calabi-type flow evolving the potential function by its Chern scalar curvature within the $\partial\bar{\partial}$-class of $\omega^{n-1}$, see also \cite{bv17, kawamura19}. Other flows of Hermitian metrics in the non-K\"ahler setting have been studied in \cite{acs,bx,ppz}. 

Assume that our Hermitian metric $\omega_0$ satisfies $\partial\bar{\partial}\omega_0^k=0$ for $k=1,2$ and let $\omega_\varphi = \omega_0 + \sqrt{-1}\partial\bar{\partial}\varphi$ for a smooth function $\varphi$ normalized so that $\int_X \varphi \omega_\varphi^n=0$. We consider a gradient flow of the Mabuchi energy defined above, starting at $\omega_0$. This flow can be expressed in terms of the potential $\varphi$ by
\begin{align}
\frac{\partial \varphi}{\partial t} = \R_\varphi + 2\Ree\langle \tr T_\varphi, \partial (\log\frac{\det g_\varphi}{\det g_0}-F)\rangle_\varphi,  \ \  \varphi(0) = 0, \label{chern-calabi_flow}
\end{align}
where $\R_\varphi = \tr_{\omega_\varphi}\Ric(\omega_\varphi)$ is the Chern scalar curvature of $\omega_\varphi$ and $\tr T_\varphi = (T_\varphi)^p_{p\bullet}$ is the trace of the torsion of $\omega_\varphi$. We note that $\underline{\R}=0$ since $c_1^{\BC}(X)=0$. When the metric is K\"ahler,  this flow agrees with the Calabi flow. We note that we immediately obtain short-time existence of this flow since its leading order term is a strictly elliptic fourth-order operator and so the equation is a fourth order quasilinear parabolic equation, following the same line of reasoning as in \cite{chen-he08}. Given that the Mabuchi energy is decreasing along this flow, we prove a new a priori estimate on the evolving metric, building on work by Chen-Cheng \cite{cc18} and the author \cite{shen19}, see Theorem \ref{bounds}. In this paper, we prove

\begin{thm}
Let $(X,\omega_0)$ be a Hermitian manifold with $c_1^{\BC}(X)=0$ and $\partial\bar{\partial}\omega_0^k=0$ for $k=1,2$. A solution to the flow given by Equation \eqref{chern-calabi_flow} starting at $\omega_0$ exists as long as the Chern scalar curvature remains bounded along the flow. In addition, if the Chern sclaar curvature remains bounded for all time, then we have smooth convergence of the flow to the unique Chern-Ricci-flat metric of the form $\omega_\infty = \omega_0 + \sqrt{-1}\partial\bar{\partial}\varphi_\infty$ for some smooth function $\varphi_\infty$ on $X$. \label{main_thm}
\end{thm}

Our assumption that $\omega_0$ satisfy $\partial\bar{\partial}\omega_0^k=0$ for $k=1,2$, is in fact equivalent to $\partial\bar{\partial}\omega_0^k=0$ for $k=1,2,\ldots,n-1$ and this condition is preserved by the flow. This assumption allows us to ensure that the volume $V=\int_X \omega_\varphi^n$ remains unchanged along the flow and to obtain $C^{3,\alpha}$ estimates for $\varphi$ along the flow dependent on a Chern scalar curvature bound. In order to show long-time existence, we will need to assume that the Chern scalar curvature remains bounded for all time along with a smoothing property that allows to obtain all higher order estimates on $\varphi$, following the work of Chen-He \cite{chen-he08} for the Calabi flow.

This paper is structured as follows:
\begin{itemize}
\item In Section 2, we cover the notation and basic properties of Hermitian metrics that we will need in the subsequent sections.
\item In Section 3, we discuss properties of the flow. We first show that the flow is indeed a gradient flow of the Mabuchi energy. We then prove that its fixed points are precisely the constant Chern scalar curvature metrics, which when $c_1^{\BC}(X)=0$, are in fact Chern-Ricci-flat metrics. Lastly, we discuss short-time existence of the flow.
\item In Section 4, we prove a priori estimates along the flow following the methods of Chen-Cheng in \cite{cc18} and previous work by the author in \cite{shen19}. Finally, we prove Theorem \ref{main_thm}.
\end{itemize}

\section{Preliminaries}

In this section, we include several well-known identities that will be needed for computations in the next sections  (see also Section 2 of \cite{tw15}). 

Let $X$ be a compact complex manifold of complex dimension $n$. We will work in complex coordinates $z^1,\ldots, z^n$ and write tensors in terms of this coordinate system. Let $g=g_{i\bar{j}}$ be a Hermitian metric on $X$ with associated $(1,1)$-form $\omega=\sqrt{-1}g_{i\bar{j}}dz^i\wedge d\overline{z^j}$ where all repeated indices are being summed from $1$ to $n$. We will also refer to $\omega$ as a Hermitian metric. 

Let $\nabla$ be the \textit{Chern connection} associated to $\omega$, defined for a $(1,0)$-form $a = a_kdz^k$ by
\begin{align}\begin{split}
&\nabla_i a_k= \partial_i a_k - \Gamma_{ik}^j a_j\ , \ \ \  \nabla_i \overline{a_k} = \partial_i \overline{a_k}\label{cov_to_partial}
\end{split}\end{align}
and for a vector field $X=X^k\partial_k$ by
\begin{align*}
\nabla_i X^k = \partial_i X^k +\Gamma^k_{ij}X^j\ , \ \ \ \nabla_i \overline{X^k} = \partial_i\overline{X^k}
\end{align*}
where $\Gamma^k_{ij} = g^{k\bar{p}}\partial_i g_{j\bar{p}}$ is the Christoffel symbol of $g$ and $g^{k\bar{p}}g_{i\bar{p}} = \delta_{ik}$. 
The Chern connection is compatible with the metric $g$, that is, $\nabla_k g_{i\bar{j}}=0$ $\forall i,j,k$. We define the \textit{trace} of a real $(1,1)$-form $\alpha=\alpha_{i\bar{j}}dz^i\wedge d\overline{z^j}$ with respect to $\omega$ by $$\tr_\omega \alpha = g^{i\bar{j}}\alpha_{i\bar{j}} = \tfrac{n\omega^{n-1}\wedge \alpha}{\omega^n}.$$

The \textit{torsion} of $g$ is defined by $T^k_{ij} = \Gamma^k_{ij} - \Gamma^k_{ji}.$ Let $\omega_\varphi = \omega+\sqrt{-1}\partial\bar{\partial}\varphi$ be another Hermitian metric on $X$. From this definition, it is clear that $$(\partial \omega)_{jk\bar{\ell}} = (\partial \omega_\varphi)_{jk\bar{\ell}}$$ where $(\partial\omega)_{jk\bar{\ell}} = \partial_j g_{k\bar{\ell}} - \partial_k g_{j\bar{\ell}}$. Denoting the torsion of $\omega_\varphi$ by $\tilde{T}$, it follows that
\begin{align}\begin{split}\label{torsion}
T^p_{jk}g_{p\bar{\ell}} = (\partial\omega)_{jk\bar{\ell}} &= (\partial\omega_\varphi)_{jk\bar{\ell}} = \tilde{T}^q_{jk}\tilde{g}_{q\bar{\ell}},
\end{split}
\end{align}
where $\tilde{g}_{i\bar{j}}$ is the metric in coordinates for $\omega_\varphi$. For simplicity, we will use the notation $$ T_{jk\bar{\ell}}=T^q_{jk}g_{q\bar{\ell}} \ , \ \ \tilde{T}_{jk\bar{\ell}}=\tilde{T}^p_{jk}\tilde{g}_{p\bar{\ell}} $$ and so the first equality in Equation \eqref{torsion} can be rewritten as $T_{jk\bar{\ell}}=\tilde{T}_{jk\bar{\ell}}$. In addition, we will let $(\tr T)_j$ denote $T^p_{pj}$. The \textit{curvature tensor} is defined by
\begin{align*}
R_{i\bar{j}k}^{\;\;\;\;\;p} = -\partial_{\bar{j}}\Gamma^p_{ik} \ , \ \ \ R_{i\bar{j}k\bar{\ell}} = g_{p\bar{\ell}}R_{i\bar{j}k}^{\;\;\;\;\;p}
\end{align*}
where we note that $\overline{R_{i\bar{j}k\bar{\ell}}} = R_{j\bar{i}\ell\bar{k}}$. We can commute the indices of the curvature tensor as follows:
\begin{align}\begin{split}\label{curv}
R_{i\bar{j}k}^{\;\;\;\;\;p} - R_{k\bar{j}i}^{\;\;\;\;\;p} = \partial_{\bar{j}}\Gamma^p_{ki} - \partial_{\bar{j}}\Gamma^p_{ik} = \partial_{\bar{j}}T^p_{ki}.
\end{split}\end{align}

The \textit{Chern-Ricci curvature} of $\omega$ is defined by $$R_{i\bar{j}} = g^{k\bar{\ell}}R_{i\bar{j}k\bar{\ell}} = -\partial_i \partial_{\bar{j}} \log\det g,$$
its associated form by $$\Ric(\omega)=\sqrt{-1}R_{i\bar{j}}dz^i\wedge d\overline{z^j}$$ and its \textit{Chern scalar curvature} by $$\R(\omega) = g^{i\bar{j}}R_{i\bar{j}} = \tr_\omega \Ric(\omega).$$

The following commutation formulae will be useful to us in the later sections. For a $(1,0)$-form $a=a_kdz^k$,
\begin{align}\begin{split}\label{a_comm_formula}
[\nabla_i, \nabla_{\bar{j}}]a_k &= -R_{i\bar{j}k\;}^{\;\;\;\;\;\ell}a_{\ell}, \ \ \ \ \ \
[\nabla_i,\nabla_{\bar{j}}]\overline{a_{l}} = R_{i\bar{j}\;\;\bar{\ell}}^{\;\;\;\bar{k}}\overline{a_{k}}\\
[\nabla_i,\nabla_j] \overline{a_k} &= -T^r_{ij}\nabla_r \overline{a_k}, \ \ \ \ \
[\nabla_{\bar{i}},\nabla_{\bar{j}}]a_k = -\overline{T^r_{ij}}\nabla_{\bar{r}}a_k
\end{split}\end{align}
and for a scalar function $f$, we have 
\begin{align}\begin{split}\label{f_comm_formula}
[\nabla_i,\nabla_j] f &= -T_{ij}^r \nabla_r f.
\end{split}\end{align}

The \textit{Chern Laplacian} with respect to $g$ of a function $f$ is defined by
\begin{align*}
\Delta f = \tr_\omega \sqrt{-1}\partial\bar{\partial}f = g^{i\bar{j}}\partial_i \partial_{\bar{j}} f = g^{i\bar{j}} \nabla_i \nabla_{\bar{j}} f  .
\end{align*}

For a complex manifold, if we assume that $
\partial\bar{\partial}\omega^k = 0 \ \text{  for  } \ k=1,2, \label{dd_assumption}
$
then in fact $\partial\bar{\partial}\omega$ vanishes for all $k=1,\ldots, n-1$, which follows from a straightforward computation. Under this assumption, the volume of the metric remains unchanged up to addition of $\partial\bar{\partial}$ of a smooth function, that is,
 \begin{align}
\int_X (\omega+\sqrt{-1}\partial\bar{\partial}\psi)^n = \int_X \omega^n, \label{volume_preservation}
\end{align}
for all $\psi$ such that $\omega+\sqrt{-1}\partial\bar{\partial}\psi>0$, hence is preserved by the flow. The condition that $\partial\bar\partial \omega_\varphi=0$ is also preserved by the flow as
\begin{align*}
\frac{d}{dt} \partial\bar\partial \omega_\varphi = \partial\bar\partial \big(\sqrt{-1}\partial\bar\partial (\R_\varphi + 2\Ree\langle \tr T_\varphi, \partial (\log\frac{\det g_\varphi}{\det g_0}-F)\rangle_\varphi)\big)=0,
\end{align*}
and similarly $\partial\bar\partial \omega_\varphi^2=0$ is preserved as
\begin{align*}
\frac{d}{dt} \partial\bar\partial \omega_\varphi^2 = 2\partial\bar\partial \omega_\varphi \wedge \frac{d}{dt}\partial\bar\partial\omega_\varphi = 0,
\end{align*}
which together gives us that $\partial\bar\partial\omega_\varphi^k=0$ along the flow if $\partial\bar\partial\omega_0^k=0$.

The Gauduchon condition, $\partial\bar{\partial}\omega^{n-1}=0$, ensures the vanishing of the integrals of Chern Laplacians of functions:
\begin{align*}
\int_X \Delta f \omega^n = n\int_X \sqrt{-1}\partial\bar{\partial} f \wedge\omega^{n-1} = n \int_X f \sqrt{-1}\partial\bar{\partial}\omega^{n-1} = 0 .
\end{align*}

We will need the following divergence theorem in the non-K\"ahler setting (see Lemma 1 of \cite{picard_cetraro}):
\begin{lemma} For any Hermitian metric $\omega$ and $V\in \Gamma(X,T^{1,0}, X)$, we have that
\begin{align*}
\int_X \nabla_i V^i \omega^n = \int_X (\tr T)_{i}V^i \omega^n,
\end{align*}
where $\nabla$ is the Chern connection with respect to $\omega$ and $(\tr T)_i = T^p_{pi}$ is the trace of the torsion of $\omega$. \label{div_thm}
\end{lemma}
Using this divergence theorem, we can show the following:
\begin{lemma}
For any metric $g$ with associated $(1,1)$-form $\omega$, under the assumption that $\omega$ is Gauduchon, i.e. $\partial\bar{\partial}\omega^{n-1}=0$, we have that 
\begin{align*}
g^{j\bar{k}}\big(\nabla_{\bar{k}}(\tr T)_{j} - \overline{(\tr T)_{k}}(\tr T)_{j}\big)=0,
\end{align*}\label{torsion_vanishing}
where $\nabla$ is the Chern connection with respect to $\omega$ and $T$ is the torsion of the Chern connection with respect to $\omega$.\label{torsion_identity} 
\end{lemma}
\begin{proof}
This identity follows directly from the Gauduchon condition. One way to see this simply is that for any smooth function $u$ on $X$, it follows from the Gauduchon condition and Lemma \ref{div_thm} that
\begin{align*}
0 &= \int_X \Delta u^2 \omega^n = 2\int_X u\Delta u \omega^n + 2\int_X g^{j\bar{k}} \nabla_j u \nabla_{\bar{k}} u \omega^n\\
&= -2\int_X g^{j\bar{k}}\nabla_j u \nabla_{\bar{k}}u \omega^n + 2\int_X g^{j\bar{k}}(\tr T)_j u \nabla_{\bar{k}} u\omega^n + 2\int_X g^{j\bar{k}} \nabla_j u \nabla_{\bar{k}} u \omega^n\\
&= 2\int_X g^{j\bar{k}}(\tr T)_j u \nabla_{\bar{k}} u\omega^n\\
&= \int_X g^{j\bar{k}}(\tr T)_j \nabla_{\bar{k}} u^2\omega^n\\
&= -\int_X (g^{j\bar{k}}(\nabla_{\bar{k}}(\tr T)_{j} - \overline{(\tr T)_{k}}(\tr T)_{j})u^2  \omega^n.
\end{align*}
Since this holds for arbitrary $u$, it follows that $g^{j\bar{k}}(\nabla_{\bar{k}}(\tr T)_{j} - \overline{(\tr T)_{k}}(\tr T)_{j})=0$. We note that if $X$ were not compact, this identity would still hold since we could take the function $u$ to be compactly supported.
\end{proof}
This outlines the key identities and formulae that we will need for the computations in this paper. Note that throughout this paper, the constants may vary from one line to another.

\section{Properties of the flow}
In this section, we prove that the flow defined in Equation \eqref{chern-calabi_flow} is indeed a gradient flow of the Mabuchi energy, that fixed points of the flow are exactly those metrics that are Chern-Ricci-flat, and that we have short-time existence of the flow.

Let $\omega_0$ be Hermitian on $X$ satisfying $\partial\bar{\partial}\omega^k=0$ for $k=1,2$. We will assume that $\omega_0$ is normalized such that $V=\int_X\omega_0^n=1$ and we note that this integral remains unchanged along the flow, see Equation \eqref{volume_preservation}. Using the assumption that $c_1^{\BC}(X)=0$, we have that $$\Ric(\omega_0) = \sqrt{-1}\partial\bar{\partial}F,$$ for a smooth function $F$ which we call the Chern-Ricci potential, normalized so that $\int_X e^F\omega_0^n = \int_X \omega_0^n$. Let us define $$\mathcal{H} = \{\omega_\varphi : \omega_\varphi=\omega_0 + \sqrt{-1}\partial\bar{\partial}\varphi >0,  \varphi \in C^\infty(X)\}.$$ We will show that the flow defined in Equation \eqref{chern-calabi_flow} is a gradient flow of the Mabuchi energy $\Mab_{\omega_0}: \mathcal{H}\rightarrow \mathbb{R}$ defined by
\begin{align*}
\Mab_{\omega_0}(\omega_\varphi) = \int_X \Big(\log\frac{\omega_\varphi^n}{\omega_0^n} - F \Big)\omega_\varphi^n + \int_X F\omega_0^n.
\end{align*}
One can check that this definition of Mabuchi energy agrees with the formula in the K\"ahler setting \cite{tian2000} (see also Section 9 of \cite{tw13}). 

To set some notation, let $\Omega:= e^F\omega_0^n$ be a volume form on $X$ which satisfies $$\sqrt{-1}\partial\bar{\partial}\log\Omega=0$$ by the fact that $F$ is the Chern-Ricci potential for $\omega_0$. In this way, we see that the Chern scalar curvature of $\omega_\varphi$ can be expressed as
\begin{align*}
\R_\varphi = -\Delta_\varphi \log\frac{\omega_\varphi^n}{\Omega}.
\end{align*}
\begin{lemma}
The Mabuchi energy is decreasing along the flow defined in Equation \eqref{chern-calabi_flow}.
\end{lemma}
\begin{proof}
Taking the time derivative of Mabuchi energy, we have that
\begin{align*}
\frac{\partial}{\partial t}\Mab_{\omega_0}(\omega_\varphi) &= \int_X \Delta_\varphi \dot{\varphi} \omega_\varphi^n + \int_X \Big(\log\frac{\omega_\varphi^n}{\omega_0^n}-F\Big)\Delta_\varphi\dot{\varphi}\omega_\varphi^n\\
&= \int_X \log\frac{\omega_\varphi^n}{e^F\omega_0^n}\Delta_\varphi\dot{\varphi}\omega_\varphi^n
\end{align*}

Let $\tilde{\nabla}$ be the Chern connection with respect to $\omega_\varphi$. Substituting $\Omega = e^F\omega_0^n$ and integrating by parts using Lemma \ref{div_thm}, we have that
\begin{align*}
\frac{\partial}{\partial t}\Mab_{\omega_0}(\omega_\varphi) &= \int_X \log\frac{\omega_\varphi^n}{\Omega}\Delta_\varphi \dot{\varphi}\omega_\varphi^n\\
&= -\int_X (g_\varphi)^{j\bar{k}}\tilde{\nabla}_j \log\frac{\omega_\varphi^n}{\Omega}\tilde{\nabla}_{\bar{k}}\dot{\varphi}\omega_\varphi^n + \int_X (g_\varphi)^{j\bar{k}}(\tr T_\varphi)_j\log\frac{\omega_\varphi^n}{\Omega} \tilde{\nabla}_{\bar{k}}\dot{\varphi}\omega_\varphi^n\\
&= \int_X (g_\varphi)^{j\bar{k}}\tilde{\nabla}_j\tilde{\nabla}_{\bar{k}}\log\frac{\omega_\varphi^n}{\Omega} \dot{\varphi}\omega_\varphi^n - \int_X (g_\varphi)^{j\bar{k}}\overline{(\tr T_\varphi)_k}\tilde{\nabla}_j \log\frac{\omega_\varphi^n}{\Omega}\dot{\varphi}\omega_\varphi^n\\
& \ \ \ \ -\int_X (g_\varphi)^{j\bar{k}}\tilde{\nabla}_{\bar{k}}((\tr T_\varphi)_j\log\frac{\omega_\varphi^n}{\Omega})\dot{\varphi}\omega_\varphi^n + \int_X (g_\varphi)^{j\bar{k}}\overline{(\tr T_\varphi)_k}(\tr T_\varphi)_j\log\frac{\omega_\varphi^n}{\Omega}\dot{\varphi}\omega_\varphi^n\\
&= -\int_X \dot{\varphi}\Big({\R_\varphi} + 2\Ree((g_\varphi)^{j\bar{k}}(\tr T_\varphi)_j\tilde{\nabla}_{\bar{k}} \log\frac{\omega_\varphi^n}{\Omega})\Big)\omega_\varphi^n,
\end{align*}
where the second last step follows from the fact that $\sqrt{-1}\partial\bar{\partial}\log\Omega=0$ and the last step uses Lemma \ref{torsion_vanishing}.
Since this flow is defined by the following evolution equation
\begin{align*}
\dot{\varphi}&= {\R_\varphi} + 2\Ree((g_\varphi)^{j\bar{k}}(\tr T_\varphi)_j\partial_{\bar{k}} \log\frac{\omega_\varphi^n}{\Omega}),
\end{align*}
we immediately have that the Mabuchi energy is decreasing along this flow.
\end{proof}

We now check that all fixed points of this flow are of constant Chern scalar curvature which, in the setting that $c_1^{\BC}(X)=0$, is equivalent to being Chern-Ricci-flat.
\begin{lemma}
A metric is a fixed point of the flow defined by Equation \eqref{chern-calabi_flow} if and only if it is Chern-Ricci-flat. \label{fixed_pts}
\end{lemma}
\begin{proof}

Assume that $\omega_\varphi$ is Chern-Ricci-flat. Then we have that 
$$0={\R_\varphi}=-{\Delta_\varphi}\log\frac{\omega_\varphi^n}{\Omega},$$ which implies that $\log\frac{\omega_\varphi^n}{\Omega}=const.$ since we are working on a compact manifold and integration by parts gives
\begin{align}
0 = \int_X \Delta u^2 \omega^n = 2 \int_X u\Delta u \omega^n + 2\int_X g^{j\bar{k}}\nabla_j u \nabla_{\bar{k}} u \omega^n = 2\int_X |\nabla u|^2_g \omega^n \Rightarrow u = const.,\label{gauduchon_identity}
\end{align}
where we used that $\omega$ is Gauduchon in the first equality. Using that $\log\frac{\omega_\varphi^n}{\Omega}=const$, we have that
\begin{align*}
\dot{\varphi}&= {\R_\varphi} + 2\Ree((g_\varphi)^{j\bar{k}}(\tr T_\varphi)_j\partial_{\bar{k}} \log\frac{\omega_\varphi^n}{\Omega})=0.
\end{align*}
 Thus, we have showed that a Chern-Ricci-flat metric is a fixed point of this flow.

In order to show the reverse direction, let us now assume that $\omega_\varphi$ is a fixed point of the flow, that is, $\dot{\varphi}=0$. This means that ${\R_\varphi}+2\Ree((g_\varphi)^{j\bar{k}}(\tr T_\varphi)_j\partial_{\bar{k}} \log\frac{\omega_\varphi^n}{\Omega})=0$. Let $\tilde{\nabla}$ be the Chern connection with respect to $\omega_\varphi$. Integrating against an arbitary smooth function $u$ and using Lemma \ref{div_thm} and Lemma \ref{torsion_identity}, we have that
\begin{align*}
0 &= \int_X u({\R_\varphi}+2\Ree((g_\varphi)^{j\bar{k}}(\tr T_\varphi)_j\tilde{\nabla}_{\bar{k}} \log\frac{\omega_\varphi^n}{\Omega}))\omega_\varphi^n\\
&= -\int_X u(g_\varphi)^{j\bar{k}}\tilde{\nabla}_j\tilde{\nabla}_{\bar{k}}\log\frac{\omega_\varphi^n}{\Omega}\omega_\varphi^n + 2\int_X u\Ree((g_\varphi)^{j\bar{k}}(\tr T_\varphi)_j\tilde{\nabla}_{\bar{k}} \log\frac{\omega_\varphi^n}{\Omega})\omega_\varphi^n\\
&= \int_X \tilde{\nabla}_j u (g_\varphi)^{j\bar{k}}\tilde{\nabla}_{\bar{k}}\log\frac{\omega_\varphi^n}{\Omega}\omega_\varphi^n + \int_X u (g_\varphi)^{j\bar{k}}\overline{(\tr T_\varphi)_k}\tilde{\nabla}_{j}\log\frac{\omega_\varphi^n}{\Omega}\omega_\varphi^n\\
&= -\int_X {\Delta_\varphi}u \log\frac{\omega_\varphi^n}{\Omega}\omega_\varphi^n + \int_X (g_\varphi)^{j\bar{k}}\tilde{\nabla}_j u \overline{(\tr T_\varphi)_k}\log\frac{\omega_\varphi^n}{\Omega}\omega_\varphi^n + \int_X u(g_\varphi)^{j\bar{k}}\overline{(\tr T_\varphi)_k}\tilde{\nabla}_j \log\frac{\omega_\varphi^n}{\Omega}\omega_\varphi^n\\
&= -\int_X {\Delta_\varphi}u \log\frac{\omega_\varphi^n}{\Omega}\omega_\varphi^n.
\end{align*}
Choosing $u=\log\frac{\omega_\varphi^n}{\Omega}$, we have by Equation \eqref{gauduchon_identity} that $$0=\int_X |\tilde{\nabla}(\log\frac{\omega_\varphi^n}{\Omega})|^2_{\varphi} \omega_\varphi^n\ge 0$$ and so $\log\frac{\omega_\varphi^n}{\Omega}=const.$ which implies that $\Ric(\omega_\varphi)=-\sqrt{-1}\partial\bar{\partial}\log\frac{\omega_\varphi^n}{\Omega} = 0$ and so $\omega_\varphi$ is Chern-Ricci-flat.
\end{proof}

We now discuss short-time existence of this flow, following from \cite{chen-he08}.
\begin{lemma}
There exists a unique solution $\varphi(t)$ satisfying Equation \eqref{chern-calabi_flow} for $t\in [0,T)$.
\end{lemma}
\begin{proof}
The flow evolves the potential function $\varphi$ by
\begin{align*}
\frac{\partial \varphi}{\partial t} &= \R_\varphi + 2\Ree\langle(\tr T_\varphi), \partial \log\frac{\det g_\varphi}{e^F\det g_0}\rangle_\varphi\\
& = -A(\nabla \varphi,\nabla^2\varphi) \varphi + f(\nabla\varphi, \nabla^2\varphi, \nabla^3\varphi).
\end{align*}
We see that this flow differs from the Calabi flow only by addition of lower order terms in $\varphi$ which can be bundled into the $f$ term, while $A$ remains a strictly elliptic operator with coefficients depending on first- and second-order derivatives of $\varphi$. Thus, the short-time existence of this flow follows the same way as for the Calabi flow, using standard parabolic theory, see \cite{chen-he08} and references therein.
\end{proof}

\section{Proof of the main theorem}

Our goal in this section is to prove a priori estimates for $\varphi$ along this flow, conditional on a bound on the Chern scalar curvature of the evolving metric. We then prove Theorem \ref{main_thm}. 

Let us begin with showing the a priori estimates:
\begin{thm}
Let $(X, \omega)$ be a compact Hermitian manifold with $c_1^{\BC}(X)=0$ and $\omega$ satisfying $\partial\bar{\partial}\omega^k = 0$ for $k=1,2$. If $\omega_\varphi= \omega +\sqrt{-1}\partial\bar{\partial}\varphi$ for a smooth potential function $\varphi$ on $X$, then for any $0<\alpha<1$, there exists $C$ depending only on $(X,\omega)$, $\|\R_\varphi\|_{C^0}$, and an upper bound for $\emph{Mab}_\omega(\omega_\varphi)$ such that $\|\varphi\|_{C^{3,\alpha}(X,\omega)}\le C.$\label{bounds}
\end{thm}

\begin{proof}
The entropy quantity $\Ent(\omega,\omega_\varphi)$ used in \cite{shen19} is implied by a bound on the  Mabuchi energy $\Mab_\omega(\omega_\varphi)$ since
\begin{align*}
\Ent(\omega,\omega_\varphi) := \int_X \log\frac{\omega_\varphi^n}{\omega^n}\omega_\varphi^n &= \Mab_\omega(\omega_\varphi) - \int_X F\omega_\varphi^n + \int_X F\omega^n\\
&\le \Mab_\omega(\omega_\varphi) + \sup_X |F|\int_X \omega_\varphi^n + C\\
&\le \Mab_\omega(\omega_\varphi) + C,
\end{align*}
where the inequalities follow from the fact that $F$ is the Chern-Ricci potential of the fixed metric $\omega$. In what follows, we will pass the dependencies of the constants on $\Ent(\omega,\omega_\varphi)$ to $\Mab_\omega(\omega_\varphi)$.

The proof follows in almost exactly the same manner as in Chen-Cheng \cite{cc18} using what has already been proven in the non-K\"ahler setting in \cite{shen19}. We include a proof here to show how we account for the torsion terms that result from handling the derivative terms of $\R_\varphi$ in the $L^\infty$ estimate that are no longer assumed to vanish. We consider the coupled second order equations
 \begin{align}\begin{split}\label{coupled_eqns}
& \ \ F = \log\tfrac{\omega_\varphi^n}{\omega^n} \\
{\Delta_\varphi}F &= -{\R_\varphi} +\tr_{\omega_\varphi} \Ric(\omega).
\end{split} \end{align}
For the $C^0$, $C^1$ and $L^p$ estimates, lifting the assumption that $\R_\varphi=const.$ and noting the fact that there is never an instance where we need to differentiate $\R_\varphi$, the estimates work out exactly the same way as in \cite{shen19} where the new bounds will additionally depend on $\|{\R_\varphi}\|_{C^0}$. It is only in the proof of the $L^\infty$ estimate that we need to differentiate ${\R_\varphi}$ and here is where the proof will somewhat differ from the constant Chern scalar curvature case (see Section 4 of \cite{shen19}). Specifically, the term of the derivative of $\Delta_\varphi F$ arising from the computation of $\Delta_\varphi|\partial F|^2_{\omega_\varphi}$ must be handled using integration by parts, and the torsion integral arising from this integration by parts will need to be controlled. 
\begin{lemma}
Let $(\varphi, F)$ be a smooth solution to \eqref{coupled_eqns}. Then there exists a constant $C$ depending only on $(X,\omega), \|{\R_\varphi}\|_{C^0}$ and $\emph{Mab}_\omega(\omega_\varphi)$, such that
\begin{align*}
\max_X (\tr_\omega \omega_\varphi) + \max_X |\partial F|^2_{\omega_\varphi} \le C .
\end{align*}
\end{lemma}
\begin{proof}
For the purposes of simplifying notation, let $\tilde{\nabla}$, $\tilde{R}$ and $\tilde{T}$ denote, respectively, the Chern connection, curvature tensor and torsion with respect to $\tilde{g}$, the associated metric to $\omega_\varphi$. Commuting derivatives using the identities in \eqref{a_comm_formula} and \eqref{f_comm_formula}, we have the following:
\begin{align}\begin{split}\label{laplacian_of_grad_norm}
{\Delta_\varphi}(|\partial F|^2_{\omega_\varphi})&=\tilde{g}^{i\bar{j}}\tilde{g}^{p\bar{q}}\big((\tilde{\nabla}_p \tilde{\nabla}_i \tilde{\nabla}_{\bar{j}}F + \tilde{T}^r_{pi}\tilde{\nabla}_r F_{\bar{j}})F_{\bar{q}} +\overline{\tilde{T}^r_{qj}}\tilde{\nabla}_i F_{\bar{r}}F_p + \tilde{\nabla}_i\overline{\tilde{T}^r_{qj}}F_{\bar{r}}F_p\\
& \ \ \ \ + (\tilde{\nabla}_{\bar{q}}\tilde{\nabla}_i \tilde{\nabla}_{\bar{j}}F + \tilde{R}_{i\bar{q}\; \bar{j}}^{\; \;\; \bar{\ell}}F_{\bar{\ell}})F_p  \big)  +  |\tilde{\nabla}\tilde{\nabla} F|^2_{\omega_\varphi}+|\tilde{\nabla} \bar{\tilde{\nabla}} F|^2_{\omega_\varphi}\\
&=\tilde{g}^{p\bar{q}}\big({\Delta_\varphi} F)_p F_{\bar{q}} + 2\text{Re}(\tilde{g}^{i\bar{j}}\tilde{g}^{p\bar{q}}\tilde{T}^r_{pi}\tilde{\nabla}_r F_{\bar{j}}F_{\bar{q}}) + \tilde{g}^{i\bar{j}}\tilde{g}^{p\bar{q}}\tilde{\nabla}_i\overline{\tilde{T}^r_{qj}}F_{\bar{r}}F_p\\
& \ \ \ \ + \tilde{g}^{p\bar{q}}({\Delta_\varphi} F)_{\bar{q}}F_p   + \tilde{g}^{i\bar{j}}\tilde{g}^{p\bar{q}}\tilde{R}_{i\bar{q}\; \bar{j}}^{\;\;\;\bar{\ell}} F_{\bar{\ell}}F_p+ |\tilde{\nabla}\tilde{\nabla} F|^2_{\omega_\varphi}+|\tilde{\nabla} \bar{\tilde{\nabla}} F|^2_{\omega_\varphi}\\
&=\tilde{g}^{p\bar{q}}\big({\Delta_\varphi} F)_p F_{\bar{q}} + 2\text{Re}(\tilde{g}^{i\bar{j}}\tilde{g}^{p\bar{q}}\tilde{g}^{r\bar{k}}\tilde{T}_{pi\bar{k}}\tilde{\nabla}_r F_{\bar{j}}F_{\bar{q}})\\
 & \ \ \ \ +\tilde{g}^{i\bar{j}}\tilde{g}^{p\bar{q}}\tilde{g}^{t\bar{r}}\tilde{\nabla}_i(\overline{\tilde{T}_{qj\bar{t}}})F_{\bar{r}}F_p + \tilde{g}^{p\bar{q}}({\Delta_\varphi} F)_{\bar{q}}F_p+ \tilde{g}^{i\bar{j}}\tilde{g}^{p\bar{q}} \tilde{g}^{k\bar{\ell}}\tilde{R}_{i\bar{q}k\bar{j}} F_{\bar{\ell}}F_p\\
 & \ \ \ \ + |\tilde{\nabla}\tilde{\nabla} F|^2_{\omega_\varphi}+|\tilde{\nabla} \bar{\tilde{\nabla}} F|^2_{\omega_\varphi}\\
\end{split}
\end{align}
We now commute indices of the curvature tensor as in \ref{curv} and trace to get that
\begin{align}\begin{split}
{\Delta_\varphi}(|\partial F|^2_{\omega_\varphi})&=2\text{Re}(\tilde{g}^{p\bar{q}}\big({\Delta_\varphi} F)_p F_{\bar{q}}) + 2\text{Re}(\tilde{g}^{i\bar{j}}\tilde{g}^{p\bar{q}}\tilde{g}^{r\bar{k}}\tilde{T}_{pi\bar{k}}\tilde{\nabla}_r F_{\bar{j}}F_{\bar{q}})\\
& \ \ \ \ +\tilde{g}^{i\bar{j}}\tilde{g}^{p\bar{q}}\tilde{g}^{t\bar{r}}\tilde{\nabla}_i(\overline{\tilde{T}_{qj\bar{t}}})F_{\bar{r}}F_p  + \tilde{g}^{p\bar{q}} \tilde{g}^{k\bar{\ell}}\tilde{R}_{k\bar{q}}F_{\bar{\ell}}F_p\\
& \ \ \ \  -\tilde{g}^{p\bar{q}}\tilde{g}^{k\bar{\ell}}\tilde{g}^{r\bar{s}}\tilde{\nabla}_{\bar{q}}(\tilde{T}_{rk\bar{s}})F_{\bar{\ell}}F_p+|\tilde{\nabla}\tilde{\nabla} F|^2_{\omega_\varphi}+|\tilde{\nabla} \bar{\tilde{\nabla}} F|^2_{\omega_\varphi} .
\end{split}\end{align}
For a general real-valued function $A(F)$, 
\begin{align*}
e^{-A(F)}{\Delta_\varphi}(e^{A(F)} |\partial F|^2_{\omega_\varphi})&= {\Delta_\varphi}(|\partial F|^2_{\omega_\varphi}) +2A'\text{Re}(\tilde{g}^{i\bar{j}}\tilde{g}^{k\bar{\ell}}(F_iF_k F_{\bar{\ell}\bar{j}} +F_iF_{\bar{\ell}}F_{k\bar{j}}))\\
& \ \ \ \ + (A'^2+A'')|\partial F|^4_{\omega_\varphi} + A'{\Delta_\varphi} F |\partial F|^2_{\omega_\varphi} ,
\end{align*}
where $F_{\bar{\ell}\bar{j}}$ denotes $\tilde{\nabla}_{\bar{j}}\tilde{\nabla}_{\bar{\ell}}F$.
Substituting \eqref{laplacian_of_grad_norm} for the first term in the above equation and using the completed square:
\begin{align*}
A'^2|\partial F|^4_{\omega_\varphi} + 2A'\text{Re}(\tilde{g}^{i\bar{j}}\tilde{g}^{k\bar{\ell}}F_iF_kF_{\bar{\ell}\bar{j}})+|\tilde{\nabla}\tilde{\nabla} F|^2_{\omega_\varphi}\ge 0 ,
\end{align*}
we obtain that
\begin{align*}
e^{-A(F)}{\Delta_\varphi}(e^{A(F)} |\partial F|^2_{\omega_\varphi})&\ge 2\text{Re}(\tilde{g}^{p\bar{q}}({\Delta_\varphi} F)_p F_{\bar{q}}) +2\text{Re}(\tilde{g}^{i\bar{j}}\tilde{g}^{p\bar{q}}\tilde{g}^{r\bar{k}}\tilde{T}_{pi\bar{k}}\tilde{\nabla}_r F_{\bar{j}}F_{\bar{q}})\\
&\ \ \ \ +\tilde{g}^{i\bar{j}}\tilde{g}^{p\bar{q}}\tilde{g}^{t\bar{r}}\tilde{\nabla}_i(\overline{\tilde{T}_{qj\bar{t}}})F_{\bar{r}}F_p + \tilde{g}^{p\bar{q}} \tilde{g}^{k\bar{\ell}}\tilde{R}_{k\bar{q}}F_{\bar{\ell}}F_p\\
& \ \ \ \ -\tilde{g}^{p\bar{q}}\tilde{g}^{k\bar{\ell}}\tilde{g}^{r\bar{s}}\tilde{\nabla}_{\bar{q}}(\tilde{T}_{rk\bar{s}})F_{\bar{\ell}}F_p + |\tilde{\nabla} \bar{\tilde{\nabla}} F|^2_{\omega_\varphi} + 2A'\tilde{g}^{i\bar{j}}\tilde{g}^{k\bar{\ell}}F_iF_{\bar{\ell}}F_{k\bar{j}}\\
& \ \ \ \ +A''|\partial F|^4_{\omega_\varphi} +A'{\Delta_\varphi} F|\partial F|^2_{\omega_\varphi}.
\end{align*}
Applying the relation $\tilde{R}_{k\bar{q}} = R_{k\bar{q}} - F_{k\bar{q}}$ to switch the Ricci curvature of $\omega_\varphi$ to that of $\omega$,
\begin{align}\begin{split}\label{af_eqn}
e&^{-A(F)}{\Delta_\varphi}(e^{A(F)} |\partial F|^2_{\omega_\varphi})\\
&\ge 2\text{Re}(\tilde{g}^{p\bar{q}}({\Delta_\varphi} F)_p F_{\bar{q}}) +2\text{Re}(\tilde{g}^{i\bar{j}}\tilde{g}^{p\bar{q}}\tilde{g}^{r\bar{k}}\tilde{T}_{pi\bar{k}}\tilde{\nabla}_r F_{\bar{j}}F_{\bar{q}})+\tilde{g}^{i\bar{j}}\tilde{g}^{p\bar{q}}\tilde{g}^{t\bar{r}}\tilde{\nabla}_i(\overline{\tilde{T}_{qj\bar{t}}})F_{\bar{r}}F_p \\
& \ \ \ \  + \tilde{g}^{p\bar{q}} \tilde{g}^{k\bar{\ell}}R_{k\bar{q}}F_{\bar{\ell}}F_p -\tilde{g}^{p\bar{q}}\tilde{g}^{k\bar{\ell}}\tilde{g}^{r\bar{s}}\tilde{\nabla}_{\bar{q}}(\tilde{T}_{rk\bar{s}})F_{\bar{\ell}}F_p + |\tilde{\nabla} \bar{\tilde{\nabla}} F|^2_{\omega_\varphi}\\
& \ \ \ \ + (2A'-1)\tilde{g}^{i\bar{j}}\tilde{g}^{k\bar{\ell}}F_iF_{\bar{\ell}}F_{k\bar{j}} +A''|\partial F|^4_{\omega_\varphi} +A'{\Delta_\varphi} F|\partial F|^2_{\omega_\varphi}\\
&\ge 2\text{Re}(\tilde{g}^{p\bar{q}}({\Delta_\varphi} F)_p F_{\bar{q}}) +2\text{Re}(\tilde{g}^{i\bar{j}}\tilde{g}^{p\bar{q}}\tilde{g}^{r\bar{k}}\tilde{T}_{pi\bar{k}}\tilde{\nabla}_r F_{\bar{j}}F_{\bar{q}})+\tilde{g}^{i\bar{j}}\tilde{g}^{p\bar{q}}\tilde{g}^{t\bar{r}}\tilde{\nabla}_i(\overline{\tilde{T}_{qj\bar{t}}})F_{\bar{r}}F_p\\
& \ \ \ \ + \tilde{g}^{p\bar{q}} \tilde{g}^{k\bar{\ell}}R_{k\bar{q}}F_{\bar{\ell}}F_p -\tilde{g}^{p\bar{q}}\tilde{g}^{k\bar{\ell}}\tilde{g}^{r\bar{s}}\tilde{\nabla}_{\bar{q}}(\tilde{T}_{rk\bar{s}})F_{\bar{\ell}}F_p + (1-(A'-\tfrac{1}{2}))|\tilde{\nabla} \bar{\tilde{\nabla}} F|^2_{\omega_\varphi}\\
& \ \ \ \ +(A'' - (A'-\tfrac{1}{2}))|\partial F|^4_{\omega_\varphi} +A'{\Delta_\varphi} F|\partial F|^2_{\omega_\varphi}  ,
\end{split}\end{align}
where we used the following Cauchy-Schwarz inequality:
\begin{align*}
(2A'-1)\tilde{g}^{i\bar{j}}\tilde{g}^{k\bar{\ell}}F_iF_{\bar{\ell}}F_{k\bar{j}} &\ge -(A'-\tfrac{1}{2})|\partial F|^4_{\omega_\varphi} - (A'-\tfrac{1}{2})|\tilde{\nabla}\bar{\tilde{\nabla}} F|^2_{\omega_\varphi}
\end{align*}
for $A'>\tfrac{1}{2}$.

In order to control the bad torsion terms (the second, third and fifth terms in the last line of \eqref{af_eqn}), we will need to strategically choose our function $A(F)$ to ensure that $1-(A'-\tfrac{1}{2})>0$ and $A''-(A'-\tfrac{1}{2})>0$. We can accomplish this by choosing $$A(F)=\kappa e^{F} +F(\tfrac{1}{2}-\varepsilon),$$ so that $A'(F)=\kappa e^{F}+\tfrac{1}{2}-\varepsilon$ and $A''(F)= \kappa e^F$. We can then pick $\varepsilon, \kappa>0$ such that

\[ 0 \le A'' - \varepsilon = A' -\tfrac{1}{2} \le \tfrac{1}{2} \ \ \ \Leftrightarrow \ \ \ \begin{cases} 
	\kappa e^{\min_X F}-\varepsilon \ge 0\\
	\kappa e^{\max_X F}-\varepsilon\le \tfrac{1}{2}\ .
   \end{cases}
\]

We can first choose $\kappa$ small enough such that $\kappa e^{\max_X F}\le \tfrac{1}{2}$. Then choose $\varepsilon$ small enough such that $\kappa e^{\min F}\ge \varepsilon$. This ensures that $A' \in(\tfrac{1}{2},1)$.

It follows that
\begin{align*}
e&^{-A(F)}{\Delta_\varphi}(e^{A(F)}|\partial F|_{\omega_\varphi}^2)\\
& \ge 2\text{Re}\big(\tilde{g}^{p\bar{q}}\big({\Delta_\varphi} F)_p F_{\bar{q}}\big)+ 2\text{Re}(\tilde{g}^{i\bar{j}}\tilde{g}^{p\bar{q}}\tilde{g}^{r\bar{k}}\tilde{T}_{pi\bar{k}}\tilde{\nabla}_r F_{\bar{j}}F_{\bar{q}}) +\tilde{g}^{i\bar{j}}\tilde{g}^{p\bar{q}}\tilde{g}^{t\bar{r}}\tilde{\nabla}_i(\overline{\tilde{T}_{qj\bar{t}}})F_{\bar{r}}F_p \\
& \ \ \ \   + \tilde{g}^{p\bar{q}} \tilde{g}^{k\bar{\ell}}R_{k\bar{q}}F_p F_{\bar{\ell}} -\tilde{g}^{p\bar{q}}\tilde{g}^{k\bar{\ell}}\tilde{g}^{r\bar{s}}\tilde{\nabla}_{\bar{q}}(\tilde{T}_{rk\bar{s}})F_{\bar{\ell}}F_p  + \tfrac{1}{2}|\tilde{\nabla} \bar{\tilde{\nabla}} F|^2_{\omega_\varphi}+ \varepsilon |\partial F|_{\omega_\varphi}^4 + A'{\Delta_\varphi} F |\partial F|_{\omega_\varphi}^2 \\
& \ge 2\text{Re}\big(\tilde{g}^{p\bar{q}}\big({\Delta_\varphi} F)_p F_{\bar{q}}\big)+ 2\text{Re}(\tilde{g}^{i\bar{j}}\tilde{g}^{p\bar{q}}\tilde{g}^{r\bar{k}}T_{pi\bar{k}}\tilde{\nabla}_r F_{\bar{j}}F_{\bar{q}}) +\tilde{g}^{i\bar{j}}\tilde{g}^{p\bar{q}}\tilde{g}^{t\bar{r}}\partial_i \overline{T_{qj\bar{t}}}F_{\bar{r}}F_p  \\
& \ \ \ \ - \tilde{g}^{i\bar{j}}\tilde{g}^{p\bar{q}}\tilde{g}^{t\bar{r}}\tilde{g}^{s\bar{k}}\partial_i \tilde{g}_{t\bar{k}}\overline{T_{qj\bar{s}}}F_{\bar{r}}F_p+ \tilde{g}^{p\bar{q}} \tilde{g}^{k\bar{\ell}}R_{k\bar{q}}F_p F_{\bar{\ell}} -\tilde{g}^{p\bar{q}}\tilde{g}^{k\bar{\ell}}\tilde{g}^{r\bar{s}}\partial_{\bar{q}}(T_{rk\bar{s}})F_{\bar{\ell}}F_p \\
& \ \ \ \  +\tilde{g}^{p\bar{q}}\tilde{g}^{k\bar{\ell}}\tilde{g}^{r\bar{s}}\tilde{g}^{i\bar{j}}\partial_{\bar{q}} \tilde{g}_{i\bar{s}}T_{rk\bar{j}}F_{\bar{\ell}}F_p  + \tfrac{1}{2}|\tilde{\nabla}\bar{\tilde{\nabla}}F|^2_{\omega_\varphi} + \varepsilon |\partial F|^4_{\omega_\varphi} - |{\Delta_\varphi} F \|\partial F|^2_{\omega_\varphi} 
\end{align*}
where we converted the covariant derivatives to partial derivatives as in \eqref{cov_to_partial} and passed the torsion terms of $\tilde{g}$ to those of $g$ as in \eqref{torsion}.

Applying Young's inequality and choosing $B$ to be at least $3(n-1)$, where the factor of $n-1$ comes from the fact that $\tr_\omega \omega_\varphi \le C(\tr_{\omega_\varphi}\omega)^{n-1}$, we have
\begin{align*}
e&^{-A(F)}{\Delta_\varphi}(e^{A(F)}|\partial F|^2_{\omega_\varphi})\\
&\ge 2\text{Re}(\tilde{g}^{p\bar{q}}\big({\Delta_\varphi} F)_p F_{\bar{q}}) + 2\text{Re}(\tilde{g}^{i\bar{j}}\tilde{g}^{p\bar{q}}\tilde{g}^{r\bar{k}}T_{pi\bar{k}}\tilde{\nabla}_r F_{\bar{j}}F_{\bar{q}})\\
& \ \ \ \  +\tilde{g}^{i\bar{j}}\tilde{g}^{p\bar{q}}\tilde{g}^{t\bar{r}}\partial_i \overline{T_{qj\bar{t}}}F_{\bar{r}}F_p  - \tilde{g}^{i\bar{j}}\tilde{g}^{p\bar{q}}\tilde{g}^{t\bar{r}}\tilde{g}^{s\bar{k}}\partial_i \tilde{g}_{t\bar{k}}\overline{T_{qj\bar{s}}}F_{\bar{r}}F_p+ \tilde{g}^{p\bar{q}} \tilde{g}^{k\bar{\ell}}R_{k\bar{q}}F_p F_{\bar{\ell}} \\
& \ \ \ \  -\tilde{g}^{p\bar{q}}\tilde{g}^{k\bar{\ell}}\tilde{g}^{r\bar{s}}\partial_{\bar{q}}(T_{rk\bar{s}})F_{\bar{\ell}}F_p   +\tilde{g}^{p\bar{q}}\tilde{g}^{k\bar{\ell}}\tilde{g}^{r\bar{s}}\tilde{g}^{i\bar{j}}\partial_{\bar{q}} \tilde{g}_{i\bar{s}}T_{rk\bar{j}}F_{\bar{\ell}}F_p + \tfrac{1}{2}|\tilde{\nabla} \bar{\tilde{\nabla}} F|^2_{\omega_\varphi} \\
& \ \ \ \ + \varepsilon |\partial F|^4_{\omega_\varphi} - |{\Delta_\varphi} F \|\partial F|^2_{\omega_\varphi} \\
&\ge  2\text{Re}(\tilde{g}^{p\bar{q}}\big({\Delta_\varphi} F)_p F_{\bar{q}}) -C(\tr_\omega \omega_\varphi)^Bg^{i\bar{j}}\tilde{g}^{k\bar{\ell}}\tilde{g}^{p\bar{q}}\partial_i \tilde{g}_{k\bar{q}}\partial_{\bar{j}}\tilde{g}_{p\bar{\ell}} + \tfrac{1}{4}|\tilde{\nabla} \bar{\tilde{\nabla}} F|^2_{\omega_\varphi}\\
& \ \ \ \ - C(\tr_\omega\omega_\varphi)^B|\partial F|^2_{\omega_\varphi}  - C(\tr_\omega\omega_\varphi)^B,
\end{align*}
where the constant in front of the fourth term depends on $\|{\R_\varphi}\|_{C^0}$.
Now, we use the following computation in the proof of Equation (9.5) of \cite{tw15} for ${\Delta_\varphi}\tr_\omega\omega_\varphi$:
\begin{align*}
{\Delta_\varphi}\tr_\omega \omega_\varphi&=\tilde{g}^{p\bar{j}}\tilde{g}^{i\bar{q}}g^{k\bar{\ell}}\nabla_k\tilde{g}_{i\bar{j}}\nabla_{\bar{\ell}}\tilde{g}_{p\bar{q}}+2\text{Re}(\tilde{g}^{i\bar{j}}g^{k\bar{\ell}}T^p_{ki}\nabla_{\bar{\ell}}\tilde{g}_{p\bar{j}})+\tilde{g}^{i\bar{j}}g^{k\bar{\ell}}T^p_{ik}\overline{T^q_{j\ell}}\tilde{g}_{p\bar{q}}\\
& \ \ \ \ +g^{i\bar{j}}F_{i\bar{j}}-R +\tilde{g}^{i\bar{j}}\nabla_i \overline{T^\ell_{j\ell}}+\tilde{g}^{i\bar{j}}g^{k\bar{\ell}}\nabla_{\bar{\ell}}T^p_{ik}-\tilde{g}^{i\bar{j}}g^{k\bar{\ell}}\tilde{g}_{k\bar{q}}(\nabla_i\overline{T^q_{j\ell}}-R_{i\bar{\ell}p\bar{j}}g^{p\bar{q}})\\
& \ \ \ \ -\tilde{g}^{i\bar{j}}g^{k\bar{\ell}}T^p_{ik}\overline{T^q_{j\ell}}g_{p\bar{q}}.
\end{align*}
Converting the first term into covariant derivatives and applying Young's inequality, we have
\begin{align*}
\tilde{g}^{p\bar{j}}\tilde{g}^{i\bar{q}}g^{k\bar{\ell}}\nabla_k \tilde{g}_{i\bar{j}}\nabla_{\bar{\ell}}\tilde{g}_{p\bar{q}} \ge \tilde{g}^{p\bar{j}}\tilde{g}^{i\bar{q}}g^{k\bar{\ell}}\partial_k \tilde{g}_{i\bar{j}}\partial_{\bar{\ell}}\tilde{g}_{p\bar{q}}-\tfrac{\varepsilon}{2}\tilde{g}^{p\bar{j}}\tilde{g}^{i\bar{q}}g^{k\bar{\ell}}\partial_k \tilde{g}_{i\bar{j}}\partial_{\bar{\ell}}\tilde{g}_{p\bar{q}} - C(\tr_\omega \omega_\varphi)^{n} .
\end{align*}
Likewise, the second term can be bounded below by
\begin{align*}
2\text{Re}(\tilde{g}^{i\bar{j}}g^{k\bar{\ell}}T^p_{ki}\nabla_{\bar{\ell}}\tilde{g}_{p\bar{j}}) \ge -\tfrac{\varepsilon}{2}\tilde{g}^{p\bar{j}}\tilde{g}^{i\bar{q}}g^{k\bar{\ell}}\partial_k \tilde{g}_{i\bar{j}}\partial_{\bar{\ell}}\tilde{g}_{p\bar{q}} - C(\tr_\omega\omega_\varphi)^{n},
\end{align*}
and the fourth term by
\begin{align*}
g^{i\bar{j}}F_{i\bar{j}}&\ge - \tfrac{|\tilde{\nabla}\bar{\tilde{\nabla}}F|_{\omega_\varphi}^2}{\delta} - C\delta(\tr_\omega \omega_\varphi)^2.
\end{align*}
It is straightforward to see that the remaining terms can be bounded below by $-C(\tr_\omega\omega_\varphi)^{n}$.
Choosing $B\ge n$ and $\delta = 4e^{-A(F)}N(B+1)(\tr_\omega\omega_\varphi)^{B}$, we arrive at the following:
\begin{align*}
{\Delta_\varphi}\tr_\omega\omega_\varphi&\ge (1-\varepsilon)\tilde{g}^{p\bar{j}}\tilde{g}^{i\bar{q}}g^{k\bar{\ell}}\partial_k\tilde{g}_{i\bar{j}}\partial_{\bar{\ell}}\tilde{g}_{p\bar{q}} -\tfrac{e^{A(F)}}{4N(B+1)(\tr_\omega\omega_\varphi)^B}|\tilde{\nabla}\bar{\tilde{\nabla}}F|^2_{\omega_\varphi} -C(\tr_\omega\omega_\varphi)^{B+2}.
\end{align*}
Observe that

\begin{align}\begin{split}
{\Delta_\varphi} (\tr_\omega \omega_\varphi)^{B+1} &= (B+1)B(\tr_\omega \omega_\varphi)^{B-1}|\partial \tr_\omega \omega_\varphi|^2_{\omega_\varphi} + (B+1)(\tr_\omega \omega_\varphi)^{B}{\Delta_\varphi} \tr_\omega \omega_\varphi\\
&\ge (B+1)(\tr_\omega \omega_\varphi)^{B}{\Delta_\varphi} \tr_\omega \omega_\varphi .\label{q_bound}
\end{split}\end{align}

Choosing $N$ sufficiently large and letting $Q:= e^{A(F)}|\partial F|^2_{\omega_\varphi} + N(\tr_\omega \omega_\varphi)^{B+1}$, we arrive at
\begin{align*}
{\Delta_\varphi} Q &= {\Delta_\varphi}(e^{A(F)}|\partial F|^2_{\omega_\varphi} + N(\tr_\omega \omega_\varphi)^{B+1})\\
& \ge  2e^{A(F)}\text{Re}(\tilde{g}^{p\bar{q}}\big({\Delta_\varphi} F)_p F_{\bar{q}}) -C(\tr_\omega \omega_\varphi)^B |\partial F|^2_{\omega_\varphi} - C(\tr_\omega \omega_\varphi)^{2B+2}\\
& \ge  2e^{A(F)}\text{Re}(\tilde{g}^{p\bar{q}}\big({\Delta_\varphi} F)_p F_{\bar{q}}) -C(\tr_\omega \omega_\varphi)^{B+1} Q  .
\end{align*}
Now we demonstrate how to handle the non-vanishing first derivative of $\Delta_\varphi F$ in the first term on the right-hand side of the above inequality. We will use integration by parts on this term in order to bound it, this follows from the same argument used by Chen-Cheng in Section 3 of \cite{cc18}. Let us compute
\begin{align*}
{\Delta_\varphi} Q^{2p+1} = (2p+1)2p|\nabla Q|^2_{\omega_\varphi}Q^{2p-1} + (2p+1)Q^{2p}{\Delta_\varphi}Q.
\end{align*}
Integrating both sides and applying our bound from Equation \ref{q_bound}, we have that
\begin{align*}
\int_X 2p|\nabla Q|^2_{\omega_\varphi}Q^{2p-1}\omega_\varphi^n &= \int_X Q^{2p}(-{\Delta_\varphi}Q)\omega_\varphi^n\\
& \le - 2\int_X e^{A(F)}\text{Re}(\tilde{g}^{p\bar{q}}\big({\Delta_\varphi} F)_p F_{\bar{q}})Q^{2p}\omega_\varphi^n\\
& \ \ \ \ + C\int_X (\tr_\omega \omega_\varphi)^{B+1} Q^{2p+1} \omega_\varphi^n.
\end{align*}
The first integral on the right-hand side can be integrated by parts using Lemma \ref{div_thm} in the following way
\begin{align*}
- 2\int_X e^{A(F)}\text{Re}(\tilde{g}^{p\bar{q}}\big({\Delta_\varphi} F)_p F_{\bar{q}})Q^{2p}\omega_\varphi^n &= 2\int_X \text{Re}(\tilde{g}^{p\bar{q}}\nabla_p(e^{A(F)}F_{\bar{q}}Q^{2p})){\Delta_\varphi}F\omega_\varphi^n\\
& \ \ \ \ -2\int_X \text{Re}(\tilde{g}^{p\bar{q}}e^{A(F)}F_{\bar{q}}Q^{2p}(\tr \tilde{T})_{p}){\Delta_\varphi}F \omega_\varphi^n
\end{align*}
Now, the first term on the right-hand side of the above equation can be handled exactly as in Equations (3.14)-(3.18) of \cite{cc18}.
Let us demonstrate how we can handle the second integral with the torsion term.

Firstly, note that $(\tr \tilde{T})_p = \tilde{T}^s_{sp} = \tilde{g}^{s\bar{\ell}}T^p_{sp}g_{p\bar{\ell}}$ and that we can bound the $\Delta_\varphi F$ term by
\begin{align*}
|{\Delta_\varphi} F|\le \|{\R_\varphi}\|_{C^0} + \|\tr_{\omega_\varphi}\Ric\|_{C^0} \le C(1+(\tr_\omega \omega_\varphi)^{n-1}).
\end{align*}

 Computing, the integral involving torsion can be controlled as follows:
\begin{align*}
-2\int_X \text{Re}(\tilde{g}^{p\bar{q}}e^{A(F)}&F_{\bar{q}}Q^{2p}(\tr \tilde{T})_{p}){\Delta_\varphi}F \omega_\varphi^n = -2\int_X \text{Re}(\tilde{g}^{p\bar{q}}e^{A(F)}F_{\bar{q}}Q^{2p}\tilde{g}^{s\bar{\ell}}T^r_{sp}g_{r\bar{\ell}}){\Delta_\varphi}F \omega_\varphi^n\\
&\le 2\int_X e^{A(F)}Q^{2p}(\tr_{\omega_\varphi}\omega)|\tilde{\nabla} F|_{\omega_\varphi}|T^r_{sp}|_{\omega_\varphi} |{\Delta_\varphi}F| \omega_\varphi^n\\
&\le \int_X e^{2A(F)}Q^{2p}|\tilde{\nabla}F|^2_{\omega_\varphi}({\Delta_\varphi}F)^2\omega_\varphi^n + \int_X Q^{2p}(\tr_{\omega_\varphi}\omega)^2|T^r_{sp}|^2_{\omega_\varphi}\omega_\varphi^n\\
&\le \int_X e^{A(F)}Q^{2p+1}({\Delta_\varphi}F)^2\omega_\varphi^n + \int_X Q^{2p}(\tr_{\omega_\varphi}\omega)^3|T^r_{sp}|^2_{\omega}\omega_\varphi^n\\
&\le C\int_X Q^{2p+1}(\tr_\omega \omega_\varphi)^{2n-2} \omega_\varphi^n + C\int_X Q^{2p+1} \omega_\varphi^n\\
&\le C\int_X Q^{2p+1}(\tr_\omega \omega_\varphi)^{2n-2} \omega_\varphi^n,
\end{align*}
where we used Young's inequality in the third line.
Now, our bound on the torsion integral is a constant multiple of the bound of the remaining terms as seen in Equation (3.20) of \cite{cc18}. The remainder of the proof for the $L^\infty$ bound on $Q$ follows from Moser iteration and several applications of the H\"older and Sobolev inequalities with respect to the reference metric $\omega$, see Section 3 of \cite{cc18} and Section 4 of \cite{cc17}. The constants and powers of the trace differ slightly from those in the K\"ahler case, but will not affect the iteration method. Lastly, to show an $L^1$ bound on the quantity $Q$ follows immediately from the fact that $|\partial F|^2_{\omega_\varphi}$ holds the same way as in (4.35) of \cite{cc17} and the $L^1$ bound on $(\tr_\omega \omega_\varphi)^{B+1}$ follows using the $L^{B+1}$ norm we obtain from the $L^p$ bound.
\end{proof}

Finally, using that we have an upper and lower bound on $\tr_\omega{\omega_\varphi}$ gives us the quasi-isometry of $\omega$ and ${\omega}_\varphi$. Now, going back to the coupled equations
\begin{align}
&F =\log \frac{\omega_\varphi^n}{\omega^n}\label{1} \\ 
\Delta_\varphi F = &-\R_\varphi +\tr_{\omega_\varphi}\Ric(\omega), \label{2}
\end{align}
by ellipticity and that the right-hand side of  \eqref{2} is bounded in $L^p$ for any $p>0$, we obtain a $W^{2,p}$ bound on $F$ for any $p>0$ (see Theorem 9.11 in \cite{gt}). By Morrey's inequality, this gives us $C^{1,\alpha}$ bounds on $F$ for any $\alpha\in(0,1)$. Using a result in the non-K\"ahler setting by \cite{tssy} and the fact that we have $C^{\alpha}$ bounds on $F$ and $\tr_\omega \omega_\varphi$, we can then obtain a $C^{2,\alpha}$ estimate on $\varphi$ for any $\alpha\in(0,1)$.  Differentiating \eqref{1}, we have that all the coefficients are bounded in $C^\alpha$ for any $\alpha\in(0,1)$, hence by Schauder estimates, we obtain $C^{3,\alpha}$ bounds on $\varphi$ for any $\alpha\in(0,1)$. This completes the proof of the theorem.
\end{proof}
From this, the result of Chen-He for the Calabi flow (\cite{chen-he08}, Theorem 3.2 and Theorem 3.3) which states that a $C^{3,\alpha}$ bound on $\varphi$ implies bounds on all higher order derivatives of $\varphi$ along the flow, can be applied here. The flow under consideration can be represented as
\begin{align*}
\frac{\partial}{\partial t}\varphi = - A(\nabla\varphi,\nabla^2\varphi)\varphi + f(\nabla\varphi, \nabla^2\varphi, \nabla^3\varphi),
\end{align*}
where $A$ is a strictly elliptic fourth-order operator with coefficients depending on first- and second-order derivatives of $\varphi$. The parabolic PDE theory used by Chen-He \cite{chen-he08} does not rely on any K\"ahler assumptions and holds in this setting as well. It follows that
\begin{lemma}
If the evolving metric $\omega_\varphi = \omega_0 + \sqrt{-1}\partial\bar\partial\varphi$ satisfies $||\varphi||_{C^{3,\alpha}(X,\omega_0)}\le C$ then in fact $||\varphi||_{C^{k,\alpha}(X,\omega_0)}\le C(k)$ for all $k>3$.
\end{lemma}

We now prove the main result of the paper:
\begin{thm}
Let $(X,\omega_0)$ be a Hermitian manifold with $c_1^{BC}(X)=0$ and $\partial\bar{\partial}\omega_0^k=0$ for $k=1,2$. A solution to the flow defined in Equation \eqref{chern-calabi_flow} exists as long as the Chern scalar curvature of the evolving metric remains bounded. In addition, if the Chern scalar curvature remains bounded for all time, then the flow converges smoothly to the unique Chern-Ricci-flat metric in the $\partial\bar{\partial}$-class of $\omega_0$. 
\end{thm}
\begin{proof}
Firstly, we know that along the flow, the Mabuchi energy of $\omega_\varphi$ with respect to $\omega_0$ is decreasing, in other words,
\begin{align*}
\frac{\partial}{\partial t}\Mab_{\omega_0}(\omega_\varphi)= -\int_X \Big( {\R_\varphi} + 2\Ree(g_\varphi^{j\bar{k}}(\tr T_\varphi)_{j}\partial_{\bar{k}} \log\frac{\omega_\varphi^n}{e^F\omega_0^n})\Big)^2\omega_\varphi^n\le 0. 
\end{align*}
In addition, the Mabuchi energy is bounded from below since
\begin{align*}
\Mab_{\omega_0}(\omega_\varphi) &= \int_X \log\frac{\omega_\varphi^n}{\omega_0^n} \omega_\varphi^n - \int_X F (\omega_\varphi^n-\omega_0^n)\\
&\ge -C - \sup_X|F| \int_X \omega_\varphi^n - C\\
&\ge -C
\end{align*}
where the first integral is bounded using the fact that the map $x\mapsto x\log x$ has a lower bound for $x>0$.
This gives us a uniform bound on the Mabuchi energy along the flow. 

Using the assumption that $\R_\varphi$ remains uniformly bounded for all time $t\in [0,\infty)$, it follows by Theorem \ref{bounds} that we have uniform $C^{3,\alpha}$ estimates on $\varphi$ for all time. By the smoothing argument, this gives us bounds on all higher order derivatives of $\varphi$. Combining the fact that we have short-time existence of the flow with these uniform bounds on the potential along the flow gives us long-time existence. Additionally, using these $C^\infty$ estimates on $\omega_\varphi$, we have that
\begin{align*}
\frac{\partial}{\partial t} \int_X \dot{\varphi}^2 \omega_\varphi^n &= \int_X 2\dot{\varphi}\ddot{\varphi}\omega_\varphi^n + \int_X \dot{\varphi}^2\Delta_\varphi \dot{\varphi}\omega_\varphi^n \le C \Big(\int_X \dot{\varphi}^2 \omega_\varphi^n \Big)^{1/2}.
\end{align*}
Let us denote $f(t) = \Big(\int_X \dot{\varphi}^2 \omega_\varphi \Big)(t)$. The above differential inequality is equivalent to
\begin{align*}
f'\le Cf^{1/2}\le C(f+\delta)^{1/2} \ \Rightarrow \ \frac{f'}{(f+\delta)^{1/2}} = ((f+\delta)^{1/2})'\le C,
\end{align*}
for a constant $\delta>0$.
For $t'>t$, integrating the above equation from $t$ to $t'$ gives us that
\begin{align*}
(f(t')+\delta)^{1/2}- (f(t)+\delta)^{1/2} \le C(t'-t).
\end{align*}
Since this holds for every $\delta>0$, we may take $\delta\rightarrow 0$ and rearrange to obtain 
\begin{align*}
f(t') \le (f(t)^{1/2}+C(t'-t))^{2},
\end{align*}
which shows that $f(t)$ can grow at most quadratically.
Since we know that the Mabuchi energy is uniformly bounded and $\frac{\partial}{\partial t} \Mab_{\omega_0}(\omega_\varphi) = - f(t)$, integration gives us that
\begin{align*}
\int_0^\infty f(t) dt < \infty.
\end{align*}
Let us assume for contradiction that $f(t)$ does not converge to $0$ for all $t$. Then, there exists an $\varepsilon\in(0,1)$ and a sequence of times $(t_i)_{i=1}^\infty\rightarrow\infty$ such that $f(t_i) \ge \varepsilon$. For each interval $[t_i-1,t_i)$,
\begin{align*}
\int_{t_i-1}^{t_i} f(t)dt &\ge \int_{\max(t_i-1, \{t \  : \ \varepsilon^{1/2}-C(t_i-t)>0\})}^{t_i} (\varepsilon^{1/2}  - C(t_i-t))^2 dt\\
&= \int_{\max(t_i-1, t_i-\frac{\varepsilon^{1/2}}{C})}^{t_i} (\varepsilon^{1/2}  - C(t_i-t))^2 dt\\
&= \frac{1}{3C}\big(\varepsilon^{3/2} - \max(0, (\varepsilon^{1/2}-C)^3)\big)\\
&= \frac{1}{3C}\min(\varepsilon^{3/2}, 3C\varepsilon + 3C\varepsilon^{1/2}+C^3)\\
&\ge C\varepsilon^{3/2}.
\end{align*}
However, since there are infinitely many such $t_i$ where $t_i\rightarrow\infty$, the integral of $f(t)$ cannot be bounded, giving us the desired contradiction.

This implies that $f(t) = \int_X \dot{\varphi}^2 \omega_\varphi^n \rightarrow 0$ for all $t\rightarrow\infty$ which implies that $\dot{\varphi}(t)\rightarrow 0$ as $t\rightarrow\infty$. Since we have $C^\infty$ estimates on $\varphi$, by the Arzela-Ascoli theorem, there exists a sequence of times $(t_j)_{j=1}^\infty$ such that $\varphi(t_j)\rightarrow \varphi_\infty$ in $C^\infty$, where $\varphi_\infty$ is smooth. In fact, since $\int_X \dot{\varphi}_\infty^2\omega_{\varphi_\infty} = 0$, $\dot{\varphi}_\infty = 0$ which by Lemma \ref{fixed_pts} implies that $\omega_\infty  = \omega_0 + \sqrt{-1}\partial\bar{\partial}\varphi_\infty$ is Chern-Ricci-flat. By the uniqueness of Chern-Ricci-flat metrics in the $\partial\bar{\partial}-$class of a metric \cite{tw09}, $\omega_\infty$ is the unique Chern-Ricci-flat metric along the flow. We can show that we indeed have $C^\infty$ convergence without passing to a sequence. To see this, suppose not. Then there exists a $k\in\mathbb{N}$, $\varepsilon>0$ and a sequence of times $(t_n)_{n=1}^\infty$ such that for all $n$,
\begin{align*}
\|\varphi(t_n) - \varphi_\infty\|_{C^k(X)} \ge \varepsilon.
\end{align*}
Since we have $C^{k+1}$ bounds on $\varphi(t_n)$, the Arzela-Ascoli theorem gives us that there exists a subsequence $(t_{n_j})_{j=1}^\infty$ such that $\varphi(t_{n_j})$ converges in $C^k$ to a limit $\varphi_\infty'$, with
\begin{align*}
\|\varphi_\infty' - \varphi_\infty\|_{C^k(X)}\ge \varepsilon.
\end{align*}
This implies that $\omega_\varphi '\neq \omega_\varphi$, but by the above argument $\omega_\varphi'$ is also Chern-Ricci-flat, and so this contradicts the uniqueness of the Chern-Ricci-flat metric in the $\partial\bar{\partial}-$class of $\omega_0$. This completes the proof of the theorem.

\end{proof}
\section*{Acknowledgements}
The author is very grateful to her thesis advisor Ben Weinkove for his helpful suggestions and his continued support and encouragement. She would also like to thank Gregory Edwards, Antoine Song and Jonathan Zhu for some useful discussions, as well as the referee for many helpful constructive comments. 

\bibliography{bib_for_all}
\end{document}